\documentclass[12pt, reqno]{amsart}
\usepackage{graphicx}
\usepackage{subfigure}
\usepackage{latexsym}
\usepackage{amsmath}
\usepackage{amssymb}
\usepackage{mathrsfs}
\usepackage{color}

 \newtheorem{theorem}{Theorem}[section]
 
 \newtheorem{Prop}[theorem]{Proposition}
 \newtheorem{Lem}[theorem]{Lemma}
 \newtheorem{Cor}[theorem]{Corollary}

\newcommand{\ba}{\begin{array}}
\newcommand{\ea}{\end{array}}
\newcommand{\beq}{\begin{equation}}
\newcommand{\eeq}{\end{equation}}

\newcommand\Let{\mathrel{\mathop:\!\!=}}

 \numberwithin{equation}{section}
 
\begin{document}

\title{Connectedness of planar self-affine sets associated with non-collinear digit sets}


\author{King-Shun Leung} \address{ Department of Mathematics and Information Technology,
The Hong Kong Institute of Education, Hong Kong}
\email{ksleung@ied.edu.hk}

\author{Jun Jason Luo} \address{Department of Mathematics,  Shantou Uinversity, Shantou 515063, China} \email{luojun@stu.edu.cn}

%
%


\thanks{The research is supported by STU Scientific Research Foundation for Talents.}

%
\date{\today}

\begin{abstract}
We study the connectedness of the planar self-affine sets
$T(A,{\mathcal{D}})$ generated by an integer expanding matrix $A$
with $|\det(A)|=3$ and a non-collinear digit set ${\mathcal D}=\{0,
v, kAv\}$ where $k\in {\mathbb Z}\setminus\{0\}$ and $v\in {\mathbb
Z}^2$ such that $\{v, Av\}$ is linearly independent. By checking the
characteristic polynomials of $A$ case by case,  we obtain a
criterion concerning only $k$ to determine the connectedness of
$T(A,{\mathcal{D}})$.
\end{abstract}

\maketitle

\bigskip

\begin{section} {\bf Introduction}
Let $M_n({\mathbb{Z}})$ denote the set of $n\times n$ matrices with
integer entries, let $A\in M_n({\mathbb{Z}})$ be expanding, i.e.,
all eigenvalues of $A$ have moduli strictly larger than $1$. Assume
$|\det A|=q$, and a finite set
${\mathcal{D}}=\{d_1,\dots,d_q\}\subset {\mathbb{R}}^n$ with
cardinality $q$, we call it a \emph{q-digit set}. It is well known
that there exists a unique \emph{self-affine set} $T\Let
T(A,{\mathcal{D}})$ \cite{LaWa} satisfying:
$$T= A^{-1}(T +
{\mathcal{D}})=\left\{\sum_{i=1}^{\infty}A^{-i}d_{j_i}: d_{j_i}\in
{\mathcal{D}}\right\}.$$ $T$ is called a \emph{self-affine tile} if
such $T$ has a nonvoid interior.

The topological structure of $T(A,{\mathcal{D}})$, especially its
connectedness,  has attracted a lot of attentions in fractal
geometry and  number theory. It was asked by Gr\"ochenig and Haas
\cite{GrHa} that given an expanding integer matrix $A\in
M_n({\mathbb{Z}})$, whether there exists a digit set ${\mathcal{D}}$
such that $T(A,{\mathcal{D}})$ is a connected tile and they
partially solved the question in ${\mathbb R}^2$. Hacon et al.
\cite{HaSaVe} proved that any self-affine tile $T(A,{\mathcal{D}})$
with a $2$-digit set is always pathwise connected.  Lau and his
collaborators (\cite{HeKiLa}, \cite{KiLa}, \cite{KiLaRa},
\cite{LeLa}) systematically studied connectedness of the self-affine
tiles generated by a kind of special digit sets of the form
$\{0,1,\dots,q-1\}v$ where $v\in {\mathbb Z}^n\setminus\{0\}$, which
are called {\em consecutive collinear digit sets}. They observed a
height reducing property (HRP) of  the characteristic polynomial of
$A$ to determine the connectedness of $T(A,{\mathcal{D}})$, and
conjectured that all monic expanding polynomials have the HRP, thus
all the tiles generated by consecutive collinear digit sets are
connected. Akiyama and Gjini \cite{AkGj} solved it up to degree $4$.
However it is still open for arbitrary degree.  In the plane, the
disk-likeness (i.e. homeomorphic to a closed unit disk) is an
interesting topic, Bandt and Gelbrich \cite{BaGe}, Bandt and Wang
\cite{BaWa}, and Leung and Lau \cite{LeLa} investigated the
disk-likeness of self-affine tiles in terms of the neighborhoods of
$T$. Deng and Lau \cite{DeLa}, and Kirat \cite{Ki} concerned
themselves about  a class of  planar self-affine tiles generated by
product digit sets.

With regard to other types of digit sets, there are few results
about the connectedness of $T(A,{\mathcal{D}})$ generated by
non-consecutive or non-collinear digit sets. In \cite{LeLu}, by
counting the neighborhoods of $T$, the authors made a first attempt
to exploit the case of non-consecutive collinear digit sets with
$|\det(A)|=3$, and obtained a complete characterization for
$T(A,{\mathcal{D}})$ to be connected or not. As a subsequent
research, in the present paper, we will focus  on the non-collinear
digit sets. Precisely, we discuss the self-affine sets
$T(A,{\mathcal{D}})$ generated by an expanding matrix $A$ with
$|\det(A)|=3$ and a non-collinear digit set ${\mathcal D}=\{0, v,
kAv\}$ for $k\in {\mathbb Z}\setminus\{0\}$ and $v\in {\mathbb R}^2$
such that $\{v, Av\}$ is linearly independent. By checking the
characteristic polynomials of $A$ case by case, we obtain a
criterion concerning only $k$ to determine the connectedness of
$T(A,{\mathcal{D}})$.

\begin{theorem}
Let $A$ be a $2\times 2$ integral expanding matrix with
$|\det(A)|=3$, and let  ${\mathcal{D}}=\{0,v, kAv\}$ be a digit set
where $k\in {\mathbb Z}\setminus\{0\}$ such that $\{v, Av\}$ is
linearly independent. Then the self-affine set $T(A, {\mathcal{D}})$
is connected if and only if $k=\pm 1$.
\end{theorem}

\end{section}

\bigskip

\begin{section}{\bf Preliminaries}

In this section,  we give some preparatory results of self-affine
sets which will be used frequently in the paper. Define
$${\mathcal{E}}=\{(d_i,d_j):\ (T+d_i)\cap(T+d_j)\ne\emptyset,\
d_i,\ d_j\in {\mathcal{D}}\}.$$   We say that $d_i$ and $d_j$ are
{\em $\mathcal{E}$-connected} if there exists a finite sequence
$\{d_{j_1},\dots,d_{j_k}\}\subset {\mathcal{D}}$ such that
$d_i=d_{j_1},d_j=d_{j_k}$ and $(d_{j_l},d_{j_{l+1}})\in
{\mathcal{E}}, 1\leq l \leq k-1.$

It is easy to check that $(d_i,d_j)\in {\mathcal{E}}$ if and only if
$$d_i-d_j=\sum_{k=1}^{\infty}A^{-k}v_k \quad \text{where}\quad  v_k\in
\Delta\mathcal{D}\Let \mathcal{D}-\mathcal{D}.$$ Then we get the
following criterion of connectedness of a self-affine set.

\begin{Prop}(\cite{Ha}, \cite{KiLa}) \label{e-connected prop}
A self-affine set $T$ with a digit set $\mathcal{D}$ is connected if
and only if any two $d_i, d_j\in {\mathcal{D}}$  are
$\mathcal{E}$-connected.
\end{Prop}

In the following, we mainly consider the planar self-affine set
$T(A,{\mathcal{D}})$ generated by a $2\times 2$ integral expanding
matrix $A$ with $|\det A|=3$ and a digit set ${\mathcal{D}}=\{0,v,
kAv\}$ such that $\{v, Av\}$ is linearly independent, where $k\in
{\mathbb Z}\setminus\{0\}$. Denote the characteristic polynomial of
$A$ by $f(x)=x^2+px+q$. Define $\alpha_i,\beta_i$ by
\begin{equation*}
A^{-i}v=\alpha_iv+\beta_iAv, \quad i=1,2,\dots.
\end{equation*}
Applying the Hamilton-Cayley theorem $f(A)=A^2+pA+qI=0$, it follows
a lemma.

\begin{Lem}(\cite{LeLa}) \label{evaluation}
Let $\alpha_i,\beta_i$ be defined as the above. Then
$q\alpha_{i+2}+p\alpha_{i+1}+\alpha_i=0$ and
$q\beta_{i+2}+p\beta_{i+1}+\beta_i=0$. Especially, $\alpha_1=-p/q,\
\alpha_2=(p^2-q)/q^2;\ \beta_1=-1/q,\ \beta_2=p/q^2$. Moreover for
$\Delta=p^2-4q\ne 0$, we have
$$\alpha_i=\frac{q(y_1^{i+1}-y_2^{i+1})}{\Delta^{1/2}}
\quad\text{and}\quad \beta_i=\frac{-(y_1^i-y_2^i)}{\Delta^{1/2}}$$
where $y_1=\frac{-p+\Delta^{1/2}}{2q},\
y_2=\frac{-p-\Delta^{1/2}}{2q}$ are the two roots of $qx^2+px+1=0$.
\end{Lem}

Let
\begin{equation*}
\tilde{\alpha}\Let
\sum_{i=1}^{\infty}|\alpha_i|,\quad
\tilde{\beta}\Let
\sum_{i=1}^{\infty}|\beta_i|.
\end{equation*}

\begin{Cor}\label{cor.ref}
Assume $f(x)=x^2+px+q$ and $g(x)=x^2-px+q$ be the characteristic
polynomials of expanding matrices $A$ and $B$, respectively. Let
$\alpha_i, \beta_i, \tilde{\alpha}, \tilde{\beta}$ for $f(x)$ be as
before; let $\alpha_i', \beta_i', \tilde{\alpha'}, \tilde{\beta'}$
be the corresponding terms for $g(x)$. Then
$$\alpha_{2j}'=\alpha_{2j}, \ \alpha_{2j-1}'=-\alpha_{2j-1}, \  \beta_{2j}'=-\beta_{2j}, \
\beta_{2j-1}'=\beta_{2j-1},$$ and hence
$\tilde{\alpha}=\tilde{\alpha'}, \ \tilde{\beta}=\tilde{\beta'}$.
\end{Cor}

When $|\det A| =3$,  it is known by \cite{BaGe} that there are $10$
eligible characteristic polynomials of $A$:
\begin{equation*}
x^2\pm 3;\quad x^2\pm x+ 3;\quad x^2\pm 2x + 3;\quad x^2\pm 3x +
3;\quad x^2\pm x - 3.
\end{equation*}

Following \cite{LeLu}, together with Corollary \ref{cor.ref},  we
obtained the estimates  or  values of the corresponding
$\tilde{\alpha}$ and $ \tilde{\beta}$ as follows:
\begin{eqnarray}
&& f(x)=x^2 \pm x+3: \quad  \tilde{\alpha}< 0.88,\ \tilde{\beta}<0.63;\label{estimate1} \\
&& f(x)=x^2 \pm 2x+3: \quad  \tilde{\alpha}< 1.17,\ \tilde{\beta}<0.73;\label{estimate2}\\
&& f(x)=x^2 \pm 3x+3: \quad  \tilde{\alpha}< 2.24,\ \tilde{\beta}<1.08;\label{estimate3}\\
&& f(x)=x^2 \pm x-3: \quad  \tilde{\alpha}=2,\
\tilde{\beta}=1.\label{estimate4}
\end{eqnarray}

\end{section}

\bigskip

\begin{section} {\bf Main results}

For a digit set ${\mathcal{D}}=\{0,v, kAv\}$, we denote by
$\Delta{\mathcal D}=\{0,\pm v, \pm(kAv-v), \pm kAv\}$ the difference
set. First we show the following simplest case according to the
characteristic polynomials of $A$.

\begin{theorem}
Let $A$  be a $2\times 2$  integral expanding matrix with
characteristic polynomial $f(x)= x^2 \pm 3$ and
${\mathcal{D}}=\{0,v, kAv\}$ be a digit set where $k\in {\mathbb
Z}\setminus\{0\}$ such that $\{v, Av\}$ is linearly independent.
Then the self-affine set $T(A, {\mathcal{D}})$ is connected if and
only if $k=\pm 1$.
\end{theorem}

\begin{proof}

Since the case of $f(x)=x^2-3$ is more or less  the same as that of
$f(x)=x^2+3$, it suffices to show the last one. If $k=1$, then
$\Delta{\mathcal D}=\{0,\pm v, \pm(Av-v), \pm Av\}$. From
$f(A)=A^2+3I=0$, we have
\begin{equation}\label{eq3.00}
I=-2A^{-2}(I+A^{-2})^{-1}=2\sum_{n=1}^{\infty}(-1)^n A^{-2n}
\end{equation} and
\begin{equation}\label{eq3.01}
v=2\sum_{n=1}^{\infty}(-1)^n
A^{-2n}v=\sum_{n=0}^{\infty}A^{-4n}\big(A^{-2}(-v)+A^{-3}(-Av)+A^{-4}v+A^{-5}(Av)\big).
\end{equation}
 Hence $v\in T-T$, or equivalently  $T\cap (T+v)\ne\emptyset$. Moreover,
\begin{equation}\label{eq3.02}
Av=\sum_{n=0}^{\infty}A^{-4n}\big(A^{-1}(-v)+A^{-2}(-Av)+A^{-3}v+A^{-4}(Av)\big)
\end{equation}
which implies $T\cap(T+Av)\ne\emptyset$. Consequently, by
Proposition \ref{e-connected prop},  $T$ is connected  (see Figure
\ref{fig1}(a)).
\begin{figure}[h]
  \centering
 \subfigure[$k=1$]{
  \includegraphics[width=5cm]{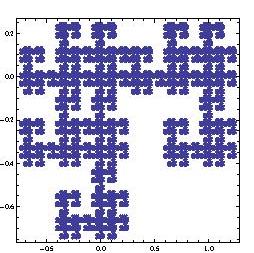}
 }
 \qquad
 \subfigure[$k=2$]{
   \includegraphics[width=5cm]{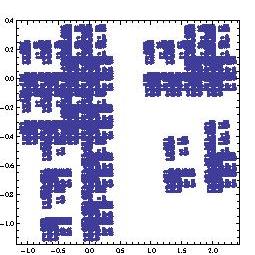}
 }
 \caption{Two cases for $f(x)=x^2+3$ and $v=(1,0)^t$.}\label{fig1}
\end{figure}

If $k=-1$, then $\Delta{\mathcal D}=\{0,\pm v, \pm(Av+v), \pm Av\}$,
and (\ref{eq3.00}), (\ref{eq3.01}), (\ref{eq3.02}) still hold. Hence
$T$ is also connected.

If  $|k|>1$, let $k_i Av+l_iv\in \Delta{\mathcal D}$ for $i\geq 1$,
then a point of $T-T$ can be written as
\begin{eqnarray*}
&&\sum_{i=1}^{\infty}A^{-i}(k_iAv+l_iv)=\sum_{i=1}^{\infty}A^{-2i}(k_{2i}Av+l_{2i}v)+
\sum_{i=1}^{\infty}A^{-2i+1}(k_{2i-1}Av+l_{2i-1}v)\\
&=& \sum_{i=1}^{\infty}(-\frac{1}{3})^{i}(k_{2i}Av+l_{2i}v)+
\sum_{i=1}^{\infty}(-\frac{1}{3})^{i}(-3k_{2i-1}v+l_{2i-1}Av)\\
&=&
\left(k_1+\sum_{i=1}^{\infty}(-\frac{1}{3})^i(l_{2i}+k_{2i+1})\right)v+
\left(\sum_{i=1}^{\infty}(-\frac{1}{3})^i(l_{2i-1}+k_{2i})\right)Av\\
&:=& Lv+KAv.
\end{eqnarray*}
As $|l_i+k_{i+1}|\leq 1+ |k|$, it follows that $|K|\leq (1+
|k|)\sum_{i=1}^{\infty}(\frac{1}{3})^i=(1+|k|)/2<|k|$. Hence
$T\cap(T+kAv)=\emptyset$ and  $(T+v)\cap(T+kAv)=\emptyset$, which
imply that $T$ is disconnected  (see Figure \ref{fig1}(b)).

\end{proof}

\begin{theorem}\label{th2}
Let $A$ be a $2\times 2$ integral expanding matrix with
characteristic polynomial $f(x)= x^2 + px \pm 3$ where $p>0$,  and
let  ${\mathcal{D}}=\{0,v, kAv\}$ be a digit set where $k\in
{\mathbb Z}\setminus\{0\}$ such that $\{v, Av\}$ is linearly
independent. Then the self-affine set $T(A, {\mathcal{D}})$ is
connected if and only if $k=\pm 1$.
\end{theorem}

\begin{proof}

For the  cases of $f(x)=x^2+px+3$ with $0<p< 3$, by using
$0=f(A)=f(A)(A-I)=A^3+(p-1)A^2+(3-p)A-3I$, we obtain
\begin{equation}\label{eq3.1}
I = \sum_{i=1}^{\infty}A^{-3i}\big((1-p)A^2-(3-p)A+2I\big).
\end{equation}

\medskip
{\bf Case 1.} $f(x)=x^2+x+3$: For $k=1$, then $\Delta{\mathcal
D}=\{0,\pm v, \pm(Av-v), \pm Av\}$. By (\ref{eq3.1}),
$I=\sum_{i=1}^{\infty}A^{-3i}\big(-2A+2I\big)$ and
\begin{equation}\label{equa3.5}
v=\sum_{i=1}^{\infty}A^{-3i}\big(-2Av+2v\big)=\sum_{i=0}^{\infty}A^{-3i}\big(A^{-2}(-v)+
A^{-3}(v-Av)+A^{-4}(Av)\big).
\end{equation}
Hence $T\cap (T+v)\ne\emptyset$. Moreover,
\begin{equation}\label{equa3.6}
Av=\sum_{i=0}^{\infty}A^{-3i}\big(A^{-1}(-v)+
A^{-2}(v-Av)+A^{-3}(Av)\big)
\end{equation}
which implies
$T\cap(T+Av)\ne\emptyset$. Consequently, $T$ is connected (see
Figure \ref{fig2}(a)).

\begin{figure}
  \centering
 \subfigure[$x^2+x+3$]{
  \includegraphics[width=5cm]{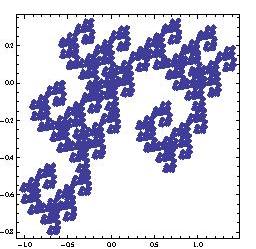}
 }
 \qquad
 \subfigure[$x^2+2x+3$]{
  \includegraphics[width=5cm]{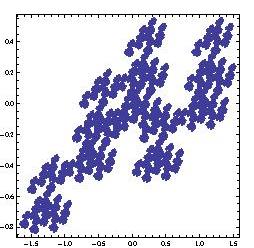}
 }\\
 \subfigure[$x^2+3x+3$]{
 \includegraphics[width=5cm]{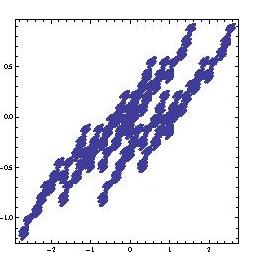}
}
 \qquad
 \subfigure[$x^2+x-3$]{
   \includegraphics[width=5cm]{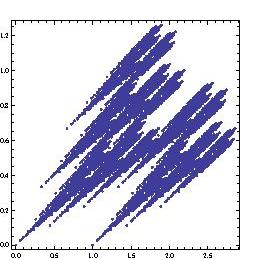}
 }
 \caption{Connected cases for $v=(1,0)^t$ and $k=1$.}\label{fig2}
\end{figure}

For $k=-1$, then $\Delta{\mathcal D}=\{0,\pm v, \pm(Av+v), \pm
Av\}$. From $f(A)=0$, we deduce that $I=(-A-2I)(A^2+I)^{-1}$, which
in turn gives
\begin{eqnarray*}
v &=& -A^{-1}v-2A^{-2}v + A^{-3}v+ 2A^{-4}v-A^{-5}v-2A^{-6}v+A^{-7}v
+ 2A^{-8}v-\cdots \\
&=& A^{-2}(-Av-v)+ A^{-3}(-Av) + A^{-4}(Av+v)+A^{-5}(Av)+
A^{-6}(-Av-v) \\
& & + A^{-7}(-Av)+ A^{-8}(Av+v) + A^{-9}(Av)+\cdots \\
&\in & T-T.
\end{eqnarray*}
Hence $T\cap (T+v)\ne \emptyset$. Multiplying the above expression
by $A$, we have
\begin{eqnarray*}
Av &=& A^{-1}(-Av-v)+ A^{-2}(-Av) + A^{-3}(Av+v)+A^{-4}(Av) \\
& & +A^{-5}(-Av-v)+ A^{-6}(-Av)+ A^{-7}(Av+v) + A^{-8}(Av)+\cdots \\
&\in & T-T
\end{eqnarray*} which implies  $T\cap (T+Av)\ne \emptyset$. It
follows that $T$ is connected.

For $|k|>1$.  A point of $T-T$ can be written as
$$\sum_{i=1}^{\infty}A^{-i}(k_iAv+l_iv)$$ where $k_i Av+l_iv\in \Delta{\mathcal D}$ for $i\geq
1$. By using the relation $A^{-i}v=\alpha_i v + \beta_i Av$,
\begin{eqnarray}\label{eq3.2}
&&\sum_{i=1}^{\infty}A^{-i}(k_iAv+l_iv)=\sum_{i=1}^{\infty}(k_iA^{-i+1}v+
l_iA^{-i}v) \nonumber\\
&=& \sum_{i=1}^{\infty}k_i(\alpha_{i-1}v + \beta_{i-1}Av) + \sum_{i=1}^{\infty}l_i(\alpha_iv + \beta_i Av)  \nonumber\\
&=&\left(k_1+ \sum_{i=1}^{\infty}(k_{i+1}+l_i)\alpha_i\right)v + \left( \sum_{i=1}^{\infty}(k_{i+1}+l_i)\beta_i \right)Av   \nonumber\\
&:=& Lv+KAv.
\end{eqnarray}
As $|l_i+k_{i+1}|\leq 1+ |k|$ and $\tilde{\beta}<0.63$
(\ref{estimate1}), we conclude $|K|\leq 0.63(1+ |k|)< |k|$, which
yields $T\cap(T+kAv)=\emptyset$ and  $(T+v)\cap(T+kAv)=\emptyset$.
Hence $T$ is disconnected  (see Figure \ref{fig3}(a)).

\begin{figure}
  \centering
 \subfigure[$x^2+x+3$]{
  \includegraphics[width=5cm]{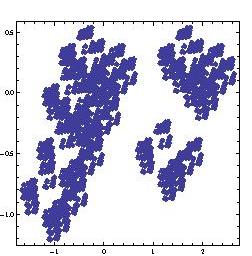}
 }
 \qquad
 \subfigure[$x^2+2x+3$]{
  \includegraphics[width=5cm]{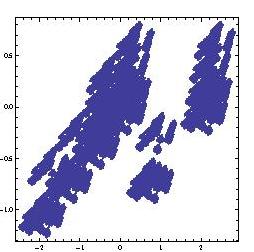}
 }\\
 \subfigure[$x^2+3x+3$]{
 \includegraphics[width=5cm]{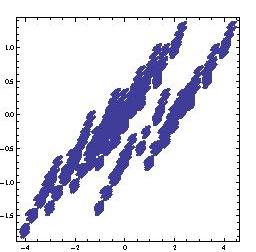}
}
 \qquad
 \subfigure[$x^2+x-3$]{
   \includegraphics[width=5cm]{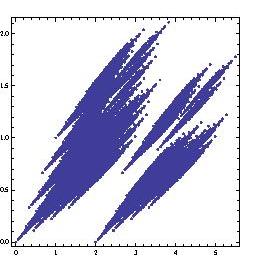}
 }
 \caption{Disconnected cases for $v=(1,0)^t$ and $k=2$.}\label{fig3}
\end{figure}

\medskip
{\bf Case 2.} $f(x)=x^2+2x+3$: For $k=1$. By (\ref{eq3.1}),
$I=\sum_{i=1}^{\infty}A^{-3i}\big(-A^2-A+2I\big)$ and
$$v=\sum_{i=1}^{\infty}A^{-3i}\big(-A^2v-Av+2v\big)=A^{-1}(-v)+\sum_{i=0}^{\infty}A^{-3i}\big(A^{-2}(-v)+ A^{-3}v+A^{-4}(Av-v)\big).$$
Hence $T\cap (T+v)\ne\emptyset$. Moreover,
$$Av=A^{-1}(-Av)+\sum_{i=0}^{\infty}A^{-3i}\big(A^{-1}(-v)+ A^{-2}v+A^{-3}(Av-v)\big)$$ which implies
$T\cap(T+Av)\ne\emptyset$. Consequently, $T$ is connected  (see
Figure \ref{fig2}(b)).

For $k=-1$. We obtain from (\ref{eq3.1}) that
\begin{eqnarray*}
v &=& -A^{-1}v-A^{-2}v+2A^{-3}v-A^{-4}v-A^{-5}v \\
&& +2A^{-6}v-A^{-7}v -A^{-8}v+2A^{-9}v+\cdots\\
&=& A^{-1}(-v)+A^{-2}(-v)+A^{-3}v +A^{-4}(Av)+A^{-5}(-Av-v) \\
&& +A^{-6}v+A^{-7}(Av)+ A^{-8}(-Av-v)+\cdots \\
&\in& T-T,
\end{eqnarray*} implying $T\cap (T+v)\ne \emptyset$. Multiplying the above expression
by $A$, we have
\begin{eqnarray*}
Av+v &=& A^{-1}(-v)+A^{-2}v +A^{-3}(Av)+A^{-4}(-Av-v)\\
&& +A^{-5}v+A^{-6}(Av)+A^{-7}(-Av-v)+\cdots\\
&\in & T-T,
\end{eqnarray*} implying $(T+v)\cap (T-Av)\ne \emptyset$. Hence $T$ is
connected.

For $|k|>1$. By (\ref{eq3.2}) and $\tilde{\beta}<0.73$
(\ref{estimate2}), we have $|K|\leq
(1+|k|)\tilde{\beta}<0.73(1+|k|)$. When $|k|\geq 3$, $|K|<
0.73(1+|k|)<|k|$, which yields $T\cap(T+kAv)=\emptyset$ and
$(T+v)\cap(T+kAv)=\emptyset$. Hence $T$ is disconnected.  When
$k=2$, suppose $(T+2Av)\cap T \ne\emptyset$ or $(T+2Av)\cap(T +
v)\ne\emptyset$, i.e., $2Av+lv\in T-T$ for $l=0$ or $-1$. By
(\ref{eq3.2}), we obtain
\begin{equation}\label{eq3.6}
(k_2+l_1)\beta_1=2-\sum_{i=2}^{\infty}(k_{i+1}+l_i)\beta_i\geq 2-
(1+2)(\tilde{\beta}-|\beta_1|)> 0.8
\end{equation}
where $\beta_1=-1/3$. It follows that  $l_1=-1,\ k_2=-2$. Then using
$f(A)=0$, we have $$A(2Av+lv)-k_1Av-l_1v=(-4+l-k_1)Av-(6+l_1)v\in
T-T.$$  It follows from $l_1=-1$ that $k_1=0$ or $2$. Hence
$-4+l-k_1\leq -4+0+0=-4$, which contradicts the inequality $|K|\leq
(1+|k|)\tilde{\beta}<3\times 0.73=2.19$. So $T$ is disconnected for
$k=2$ (see Figure \ref{fig3}(b)).

 When $k=-2$, suppose  $2Av+lv\in T-T$ where $l=0$ or $1$. Similarly, it yields from (\ref{eq3.6}) that
$l_1=-1,\ k_2=-2$ and $k_1=-2$ or $0$. If $k_1=-2$, then
$(l-2)Av-5v\in T-T$. Multiplying the expression  by $A$ and using
$f(A)=0$, we obtain
$$(l-2)A^2v-5Av-(k_2Av+l_2v)=(1-2l)Av+(6-3l-l_2)v\in T-T$$ and
$6-3l-l_2\geq 6-3-1=2$. On the other hand, by (\ref{eq3.2}) and
(\ref{estimate2}),  $-5.81=-2-3\tilde{\alpha}\leq L\leq
-2+3\tilde{\alpha}<1.81$. This is ridiculous. If $k_1=0$, then
$(l-4)Av-5v\in T-T$ and $|l-4|\geq 3$ contracts $|l-4|=|K|< 2.19$.
Therefore $T$ is disconnected for $k=-2$.

\medskip
{\bf Case 3.} $f(x)=x^2+3x+3$:  For $k=1$. From $(A-I)f(A)=0$, we
get $A^2+A-I=2(A+I)^{-1}$, which yields
$I=-A^{-1}+A^{-2}+2\sum_{i=3}^{\infty}(-1)^{i+1}A^{-i}$ and
\begin{eqnarray*}
v&=&-A^{-1}v+A^{-2}v+2\sum_{i=3}^{\infty}(-1)^{i+1}A^{-i}v\\
&=&A^{-2}(v-Av)+ A^{-3}v +
\sum_{i=0}^{\infty}A^{-2i}\big(A^{-4}(Av-v)+ A^{-5}(v-Av)\big).
\end{eqnarray*}
 Hence $T\cap (T+v)\ne\emptyset$. Moreover,
$$Av=A^{-1}(v-Av)+ A^{-2}v + \sum_{i=0}^{\infty}A^{-2i}\big(A^{-3}(Av-v)+ A^{-4}(v-Av)\big)$$ which implies
$T\cap(T+Av)\ne\emptyset$. Consequently, $T$ is connected  (see
Figure \ref{fig2}(c)).

For $k=-1$. From $f(A)=0$ we get $A+2I=-(A+I)^{-1}$. It follows that
$I= -2A^{-1}-A^{-2}+A^{-3}-A^{-4}+A^{-5}+\cdots$ and
\begin{eqnarray*}
v &=& -2A^{-1}v-A^{-2}v+A^{-3}v-A^{-4}v+A^{-5}v+\cdots \\
&=& A^{-1}(-v) +A^{-2}(-Av-v)+A^{-3}v+
A^{-4}(-v)+A^{-5}v+A^{-6}(-v)+\cdots \\
&\in & T-T,
\end{eqnarray*} then $T\cap (T+v)\ne\emptyset$. Also we can deduce
immediately that
\begin{eqnarray*}
Av+v &=& A^{-1}(-Av-v)+ A^{-2}v+ A^{-3}(-v)+A^{-4}v +
A^{-5}(-v)+\cdots\\
&\in& T-T,
\end{eqnarray*} which yields $(T+v)\cap (T-Av)\ne\emptyset$. As a
result, $T$ is connected.

For $k>1$. By (\ref{eq3.2}) and (\ref{estimate3}), we have $|L|\leq
k+(1+k)\tilde{\alpha}<2.24+3.24k$.  Suppose $kAv+lv\in T-T$ for
$l=0$ or $-1$. Multiplying (\ref{eq3.2}) by $A$ and then subtracting
$k_1Av + l_1v$ from both sides, we see that
\begin{equation}\label{equa3.9}
(-3k+l-k_1)Av -(3k+l_1)v\in T-T. \end{equation} Repeating the
process, we obtain
\begin{equation}\label{equa3.10}
(6k-3l+3k_1-l_1-k_2)Av + (9k-3l+3k_1-l_2)v\in T-T.
\end{equation}
Since $9k-3l+3k_1-l_2\geq 9k-0-3k-1=6k-1 > 2.24+3.24k$, which
exceeds the upper bound of $|L|$. It concludes that
$T\cap(T+kAv)=\emptyset$ and $(T+v)\cap(T+kAv)=\emptyset$, that is,
$T$ is disconnected (see Figure \ref{fig3}(c)).

For $k\leq -3$, then  $|L|\leq -k+(1-k)\tilde{\alpha}<2.24-3.24k$.
It follows from (\ref{equa3.10}) that $|9k-3l+3k_1-l_2|\geq
-9k-3+3k-1=-6k-4 > 2.24-3.24k$, which also exceeds the upper bound
of $|L|$.

For $k=-2$, then $|K|< 3.24$ and $|L|< 8.72$.  From (\ref{equa3.9}),
we have $6+l-k_1< 3.24$, hence $l=-1$ and
 $k_1=2$.  From (\ref{equa3.10}), we have $l_2=-1$ and $3+l_1+k_2<
 3.24$, it follows that   $l_1=0, k_2=-2$, or $l_1=0, k_2=0$, or $l_1=1,
 k_2=-2$.

 When $l_1=0, k_2=-2$. Multiplying (\ref{equa3.10}) by $A$ and then subtracting
$k_3Av + l_3v$, we get $$(-5-k_3)Av+(3-l_3)v\in T-T.$$ By
$|5+k_3|\leq 3.24$, it yields $k_3=-2$. Repeating this process, we
obtain $$(12-l_3-k_4)Av+(9-l_4)v\in T-T.$$ Hence
 we get a contradiction $|12-l_3-k_4|\geq 9>3.24$.

 When $l_1=0, k_2=0$. Multiplying (\ref{equa3.10}) by $A$ and then subtracting
$k_3Av + l_3v$, we get $$(1-k_3)Av+(9-l_3)v\in T-T.$$ It yields
$l_3=1$ and $k_3=0$ or $2$.  Repeating this process, we obtain
$$(3k_3+5-k_4)Av+(3k_3-3-l_4)v\in T-T.$$ If $k_3=0$, then $(5-k_4)Av+(-3-l_4)v\in
T-T$, and $k_4=2$. Finally we get $(-12-l_4-k_5)Av+(-9-l_5)v\in T-T$
and a contradiction $|12+l_4+k_5|\geq 9>3.24$.  If $k_3=2$, then
$(11-k_4)Av+(3-l_4)v\in T-T$, and also $|11-k_4|\geq 9>3.24$.

 When $l_1=1, k_2=-2$. By the same argument as above,  we first get $$(-2-k_3)Av+(6-l_3)v\in
T-T.$$ It yields $k_3=0$.  Repeating this process, we obtain
$(12-l_3-k_4)Av+(6-l_4)v\in T-T$ and $|12-l_3-k_4|\geq 9>3.24$
follows. Therefore $T$ is disconnected for $|k|>1$.

\medskip
{\bf Case 4.} $f(x)=x^2+x-3$: For $k=1$. We deduce from $f(A)=0$
that $I= (A^2-I)^{-1}(-A+2I)$, which yields
$I=\sum_{i=0}^{\infty}A^{-2i}(-A^{-1}+2A^{-2})$ and
\begin{eqnarray*}
v&=&\sum_{i=0}^{\infty}A^{-2i}(-A^{-1}+2A^{-2})v\\
&=&A^{-2}(v-Av)+ \sum_{i=1}^{\infty}A^{-2i}\big(A^{-1}(Av-v)+
A^{-2}v\big).
\end{eqnarray*}
 Hence $T\cap (T+v)\ne\emptyset$. Moreover,
$$Av=A^{-1}(v-Av)+ \sum_{i=1}^{\infty}A^{-2i}\big((Av-v)+
A^{-1}v\big)$$ which implies $T\cap(T+Av)\ne\emptyset$.
Consequently, $T$ is connected  (see Figure \ref{fig2}(d)).

For $k=-1$. It follows from $v=A^{-1}(Av)\in T-T$ that
 $T\cap (T+v)\ne \emptyset$. Moreover, we can get $A+I=
-I+(A-I)^{-1}$ from $f(A)=0$. This implies
\begin{eqnarray*}
Av+v &=& A^{-1}(-Av)+A^{-2}(Av)+A^{-3}(Av)+A^{-4}(Av)+\cdots\\
&\in & T-T,
\end{eqnarray*} that is, $(T+v)\cap (T-Av)\ne \emptyset$. Hence $T$
is connected.

For $k>1$. By (\ref{eq3.2}) and (\ref{estimate4}), we have $|K|\leq
(1+k)\tilde{\beta}=1+k$ and  $|L|\leq k+(1+k)\tilde{\alpha}=2+3k$.
Suppose $kAv+lv\in T-T$ for $l=0$ or $-1$. Multiplying (\ref{eq3.2})
by $A$ and then subtracting $k_1Av + l_1v$ from both sides, we have
$$(-k+l-k_1)Av + (3k-l_1)v\in T-T.$$  Repeating the process, we obtain
\begin{equation}\label{equa3.11}
(4k-l+k_1-l_1-k_2)Av + (-3k+3l-3k_1-l_2)v\in T-T.
\end{equation}
Note $4k-l+k_1-l_1-k_2\geq 2k-1$.  When $k\geq 3$, $2k-1>k+1$ which
contradicts the upper bound of $|K|$, hence $T$ is disconnected;
when $k=2$, it forces $l=0,\ k_1=-2,\ l_1=1,\ k_2=2 $ and
$3Av-l_2v\in T-T$, similarly which implies $$(-3-l_2-k_3)Av+
(9-l_3)v\in T-T.$$ It is required that $|9-l_3|\leq 8$, hence
$l_3=1$. Furthermore, from $(-3-l_2-k_3)Av+ 8v\in T-T$, we can
deduce that
$$(11+l_2+k_3-k_4)Av + (-9-3l_2-3k_2-l_4)v\in T-T.$$ Since
$11+l_2+k_3-k_4\geq 6 >3$, we also get a contradiction, and $T$ is
disconnected. Consequently, $T$ is disconnected for all $k>1$ (see
Figure \ref{fig3}(d)).

For $k<-1$, then $|K|\leq 1-k$. From (\ref{equa3.11}), it follows
that $|K|\geq -4k+l-k_1+l_1+k_2\geq -4k-1+k+1+k=-2k>1-k$, which is
impossible. Therefore   $T$ is disconnected for $|k|>1$.

\end{proof}

\begin{theorem}\label{th3}
Let $B\in M_2({\mathbb{Z}})$ be an expanding integral matrix with
characteristic polynomial $g(x)= x^2-px\pm 3$ where $p> 0$, and let
${\mathcal{D}}=\{0,v, kBv\}$ be a digit set where $k\in {\mathbb
Z}\setminus\{0\}$ such that $\{v, Bv\}$ is linearly independent.
Then the self-affine set $T(B, {\mathcal{D}})$ is connected if and
only if $k=\pm 1$.
\end{theorem}

\begin{proof}
The characteristic polynomial of $A$ is $f(x)=x^2+px+q$ if and only
if that of $-A$ is $g(x)=x^2-px+q$.  Since only the radix expansions
matter, we may assume $B=-A$.  The proof for the disconnectedness of
$T(B, {\mathcal{D}})$ when $|k|>1$ can be adapted easily from the
proof of Theorem \ref{th2} by applying Corollary \ref{cor.ref}. Let
${\mathcal{D}}=\{0,v, Av\}=\{0,v, -Bv\}$ and ${\mathcal{D}}'=\{0,v,
-Av\}=\{0,v, Bv\}$. For $|k|=1$, we deduce the connectedness of
$T_1:= T(B, {\mathcal{D}})$ (respectively, $T_1'=T(B,
{\mathcal{D}}')$) from that of $T=T(A, {\mathcal{D}})$
(respectively, $T'=T(A, {\mathcal{D}}')$). We only show the case of
$f(x)= x^2+x+3$. In Case 1 of the proof of Theorem \ref{th2}, from
(\ref{equa3.5}) we have \begin{equation*}
v=\sum_{i=0}^{\infty}(-B)^{-3i}\big(B^{-2}(-v)-B^{-3}(v+Bv)+B^{-4}(-Bv)\big)\in
T_1-T_1;
\end{equation*} and from (\ref{equa3.6}) we have
\begin{equation*}
Bv=\sum_{i=0}^{\infty}(-B)^{-3i}\big(B^{-1}(-v)-B^{-2}(v+Bv)+B^{-3}(-Bv)\big)\in
T_1-T_1.
\end{equation*} Hence $T_1$ is connected. Similarly, it can be
verified that $v, Bv\in T_1'-T_1'$ and  $T_1'$ is also connected.
\end{proof}

The above proof is indeed an application of the following more
general result.

\begin{theorem}
Let $A, B\in M_2({\mathbb{Z}})$ be two expanding integral matrices
with characteristic polynomials $f(x)= x^2+px+q$ and $g(x)=
x^2-px+q$ respectively. Let $v,w$ be two non-zero  vectors such that
the two sets $\{v, Av\}$ and $\{w, Bw\}$ are both linearly
independent. Denote by $L$ and $M$ the lattices generated by $\{v,
Av\}$ and $\{w, Bw\}$ respectively.  Let ${\mathcal{D}}=\{c_i
Av+d_iv\in L: i=0,1,\dots, |q|-1\}$  and  ${\mathcal{D}}'=\{-c_i
Bw+d_iw\in M: i=0,1,\dots, |q|-1\}$ be two digit sets.  Then
$T_1=T(A, {\mathcal{D}})$ is connected if and only if $T_2=T(B,
{\mathcal{D}}')$ is connected.
\end{theorem}

\begin{proof}
We may assume $B=-A$ and $w=-v$. Then ${\mathcal{D}}={\mathcal{D}}'$
and $L=M$. Consider $a_iAv+b_iv=-a_iBv+b_iv\in \Delta
{\mathcal{D}}=\Delta {\mathcal{D}}'$, it follows that $a_iAv+b_iv=
A^{-1}(c_1'Av+d_1'v)+ A^{-2}(c_2'Av+d_2'v)+\cdots\in T_1-T_1$ if and
only if $-a_iBv+b_iv= B^{-1}(c_1'Bv-d_1'v)+
B^{-2}(-c_2'Bv+d_2'v)+\cdots\in T_2-T_2$. Hence $c_iAv+d_iv,
c_jAv+d_jv\in {\mathcal D}$ are ${\mathcal E}$-connected if and only
if $-c_iBv+d_iv, -c_jBv+d_jv\in {\mathcal D}'$ are ${\mathcal
E}$-connected, then the theorem is proved by Proposition
\ref{e-connected prop}.
\end{proof}

We can also deduce from Theorems \ref{th2} and \ref{th3} that

\begin{Cor}
Let $A$ be a $2\times 2$ integral expanding matrix with
characteristic polynomial $f(x)=x^2+px+q$ where $|q|=3$. Let $v\in
{\mathbb R}^2$ such that $\{v, Av\}$ is linearly independent. Then
the self-affine set $T(A, {\mathcal{D}})$ is connected for
${\mathcal{D}}=\{0, v, Av+v\}$ or $\{0, v, -Av+v\}$.
\end{Cor}

\begin{proof}
Notice that the difference set $\Delta\{0,v, Av+v\}=\{0, \pm v, \pm
Av, \pm(Av+v)\}=\Delta\{0,v, -Av\}$ and $\Delta\{0,v, -Av+v\}=\{0,
\pm v, \pm Av, \pm(Av-v)\}=\Delta\{0,v, Av\}$. Hence the result
follows from Proposition \ref{e-connected prop}.
\end{proof}

\bigskip

\noindent{\bf Remarks:} The connectedness of self-affine sets or
self-affine tiles is far from known extensively. Even  for the
planar case,  there are still a lot of unsolved questions.  The
following may be some interesting topics related to the paper.

\medskip

\noindent{\bf Q1.}  Can we characterize the connectedness of  $T(A,
{\mathcal D})$ with $|\det(A)|=3$ and  ${\mathcal D}=\{0, v,
kAv+lv\}$?

\medskip

\noindent{\bf Q2.}  For a two dimensional digit set ${\mathcal
D}=\{0, v,\dots, (l-1)v, Av,\dots, kAv\}$ with $l+k=|\det(A)|>3$,
can we apply the same method to study the connectedness of $T(A,
{\mathcal D})$?

\end{section}

\bigskip
\bigskip

\subsection*{Acknowledgments:} The authors would like to thank
Professor Ka-Sing Lau for suggesting a related question and advice
on the work.

\end{document}